\documentclass{article}
\usepackage[utf8]{inputenc}
\usepackage[T1]{fontenc}
\usepackage[english]{babel}
\usepackage{hyperref}
\usepackage{bookmark}
\usepackage{amsmath}
\usepackage{amsfonts}
\usepackage{amssymb}
\usepackage{amsthm}
\usepackage{array}
\usepackage{fancyhdr}
\usepackage{graphicx}
\usepackage{fancybox}
\usepackage{xcolor}
\usepackage{algorithm}
\usepackage{algorithmic}
\usepackage{color}
\usepackage{subfigure}
\usepackage{tikz}
\usepackage{pgf,tikz}
\usepackage{mathrsfs}
\usetikzlibrary{arrows}

\def\tp{\ensuremath{\hdots}}
\def\infini{\ensuremath{\infty}}

\def\E{\ensuremath{\mathbb{E}}}

\def\N{\ensuremath{\mathbb{N}}}

\def\P{\ensuremath{\mathbb{P}}}

\def\R{\ensuremath{\mathbb{R}}}

\def\Z{\ensuremath{\mathbb{Z}}}

\def\aM{\text{argmax}}

\newcommand{\psh}[2]{\ensuremath{\left\langle #1,#2\right\rangle}}

\definecolor{qqffqq}{rgb}{0.,1.,0.}
\definecolor{ttqqqq}{rgb}{0.2,0.,0.}
\definecolor{ffqqqq}{rgb}{1.,0.,0.}
\definecolor{qqqqff}{rgb}{0.,0.,1.}
\definecolor{xdxdff}{rgb}{0.49019607843137253,0.49019607843137253,1.}
\definecolor{cqcqcq}{rgb}{0.7529411764705882,0.7529411764705882,0.7529411764705882}

\usepackage{algorithm}
\usepackage{algorithmic}
\newtheorem{theorem}{Theorem} 
\newtheorem{proposition}{Proposition} 

\title{Inconsistency of Template Estimation with the Fréchet mean in Quotient Space}
\author{Loïc Devilliers\footnote{Université Côte d’Azur, Inria, France \href{loic.devilliers@inria.fr}{loic.devilliers@inria.fr}} \and Xavier Pennec\footnote{Université Côte d’Azur, Inria, France} \and Stéphanie Allassonnière\footnote{Université Paris Descartes, INSERM UMRS 1138, Centre de Recherche des Cordeliers, France}}

\begin{document}

\maketitle

\begin{abstract}
We tackle the problem of template estimation when data have been randomly transformed under an isometric group action in the presence of noise. In order to estimate the template, one often minimizes the variance when the influence of the transformations have been removed (computation of the Fréchet mean in quotient space). 
The consistency bias is defined as the distance (possibly zero) between the orbit of the template and the orbit of one element which minimizes the variance.
In this article we establish an asymptotic behavior of the consistency bias with respect to the noise level. This behavior is linear with respect to the noise level. As a result the inconsistency is unavoidable as soon as the noise is large enough. In practice, the template estimation with a finite sample is often done with an algorithm called max-max. We show the convergence of this algorithm to an empirical Karcher mean.
Finally, our numerical experiments show that the bias observed in practice cannot be attributed to the small sample size or to a convergence problem but is indeed due to the previously studied inconsistency.

\end{abstract}
\newpage
\tableofcontents

\section{Introduction}
The template estimation is a well known issue in different fields such as statistics on signals~\cite{kur}, shape theory, computational anatomy~\cite{gui,jos,coo} etc. In these fields, the template (which can be viewed as the prototype of our data) can be (according to different vocabulary) shifted, transformed, wrapped or deformed due to different groups acting on data. Moreover, due to a limited precision in the measurement, the presence of noise is almost always unavoidable. These mixed effects on data lead us to study the consistency of algorithms which claim to compute the template. A popular algorithm consists in the minimization of the variance, in other words, the computation of the Fréchet mean in quotient space. This method has been already proved to be inconsistent~\cite{big,mio,dev}. One way to avoid the inconsistency is to use another framework, for a instance a Bayesian paradigm~\cite{che}. However, if one does not want to change the paradigm, then one needs to have a better understanding of the geometrical and statistical origins of the inconsistency.

\textbf{Notation:} in this paper, we suppose that observations belong to a Hilbert space $(H,\: \psh{\cdot}{\cdot})$, we denote by $\|\cdot\|$ the norm associated to the dot product $\psh{\cdot}{\cdot}$. We also consider a group of transformation $G$ which acts isometrically on $H$ the space of observations. This means that $x\mapsto g\cdot x$ is a linear automorphism of $H$, such that\footnote{Note that in this article, $g\cdot x$ is the result of the action of $g$ on $x$, and $\cdot$ should not to be confused with the multiplication of real numbers noted $\times$.} $\|g\cdot x\|=\|x\|$, $g'\cdot (g\cdot x)=(g'g)\cdot x$ and $e\cdot x=x$ for all $x\in H$, $g,\: g'\in G$, where $e$ is the identity element of $G$. 

\textbf{The generative model} is the following: we transform an unknown template~$t_0\in H$ with $\phi$ a random and unknown element of the group $G$ and we add some noise $\sigma\epsilon$ with a positive noise level $\sigma$, $\epsilon$ a standardized noise: $\E(\epsilon)=0$, $\E(\|\epsilon\|^2)=1$. Moreover we suppose that $\epsilon$ and $\phi$ are independent random variables. Finally, the only observable random variable is:
\begin{equation}Y=\phi\cdot t_0+\sigma\epsilon\label{modgen}.\end{equation} 
If we assume that the noise is independent and identically distributed on each pixel or voxel with a standard deviation $s$, then $\sigma=\sqrt{N}s$, where $N$ is the number of pixels/voxels.

\textbf{Quotient space and Fréchet mean:} the random transformation of the template by the group leads us to project the observation $Y$ into the quotient space defined as the set containing all the orbit $[x]=\{g\cdot x,\: g\in G\}$ for $x\in H$. Because the action is isometric, the quotient space $H/G$ is equipped with a pseudometric\footnote{$d_Q$ is called a pseudometric because $d_Q([x],[y])$ can be equal to zero even if $[x]\neq [y]$. If the orbits are closed sets then $d_Q$ is a distance.} defined by:
\begin{equation*}
d_Q([x],[y])=\underset{g\in G}{\inf} \|x-g\cdot y\|=\underset{g\in G}{\inf} \|g\cdot x-y\|.
\end{equation*}
The quotient pseudometric is the distance between $x$ and $y'$ where $y'$ is the registration of $y$ with respect to $x$.
We define the variance of the random orbit~$[Y]$ as the expectation of the square pseudometric between the random orbit~$[Y]$ and the orbit of a point~$x$ in $H$:
\begin{equation}
F(x)=\E(d_Q^2([x],[Y]))=\E(\inf_{g\in G}\|x-g\cdot Y\|^2)=\E(\inf_{g\in G}\|g\cdot x-Y\|^2).
\label{var}
\end{equation}
Note that $F(x)$ is well defined for all $x\in H$ because $\E(\|Y\|^2)$ is finite. In order to estimate the template, one often minimizes this function. If $m_\star\in H$ minimizes $F$, then $[m_\star]$ is called a Fréchet mean of $[Y]$. The consistency bias, noted $CB$, is the pseudometric between the orbit of the template $[t_0]$ and $[m_\star]$: $CB=d_Q([t_0],[m_\star])$. If such a $m_\star$ does not exist, then the consistency bias is infinite.

\textbf{Questions:} 
\begin{itemize}
\item What is the behavior of the consistency with respect to the noise?
\item How to perform such a minimization of the variance? Indeed, in practice we have only a sample and not the whole distribution.
\end{itemize}

\textbf{Contribution:} in this article, we provide a Taylor expansion of the consistency bias when the noise level $\sigma$ tends to infinity. As we do not have the whole distribution, we minimize the empirical variance given a sample. An element which minimizes the variance is called an empirical Fréchet mean. We already know that the empirical Fréchet mean converges to the Fréchet mean when the sample size tends to infinity~\cite{zie}. Therefore our problem is reduced to find an empirical Fréchet mean with a finite but sufficiently large sample. One algorithm called the max-max algorithm~\cite{all} aims to compute such an empirical Fréchet mean.
We establish some properties of the convergence of this algorithm. In particular, when the group is finite, the algorithm converges in a finite number of steps to an empirical Karcher mean (a local minimum of the empirical variance given a sample). This helps us to illustrate the inconsistency in this very simple framework. 

Of course, generally people use a subgroup of diffeomorphisms which acts non isometrically on data such that images, landmarks etc. We believe that studying the inconsistency in this simplified framework will help us to better understand more complex situations. Moreover it is also possible to define and use isometric actions on curves~\cite{hit,kur} or on surfaces~\cite{kur2} where our work can be directly applied. 

This article is organized as follows: in Section~\ref{sec:inc}, we study the presence of the inconsistency and we establish the asymptotic behavior when the noise parameter $\sigma$ tends to $\infini$. In Section~\ref{sec:max} we detail the max-max algorithm and its properties. Finally, in Section~\ref{sec:sim} we illustrate the inconsistency with synthetic data.
\section{Inconsistency of the Template Estimation}
\label{sec:inc}
We start with the main theorem of this article which gives us an asymptotic behavior of the consistency bias when the noise level $\sigma$ tends to infinity. One key notion in Theorem~\ref{theo1} is the concept of fixed point under the action $G$: a point $x\in H$ is a fixed point if for all $g\in G, \: g\cdot x=x$. We require that the support of the noise $\epsilon$ is not included in the set of fixed points. But this condition is almost always fulfilled. For instance in $\R^n$ the set of fixed points under a linear group action is a null set for the Lebesgue measure (unless the action is trivial $g\cdot x=x$ for all $g\in G$ but this situation is irrelevant).  

\begin{theorem}
\label{theo1}
Let us suppose that the support of the noise $\epsilon$ is not included in the set of fixed points under the group action. Let $Y$ be the observable variable defined in Equation~\eqref{modgen}. If the Fréchet mean of $[Y]$ exists, then we have the following lower and upper bounds of the consistency bias noted $CB$:
\begin{equation}
\sigma K-2\|t_0\|\leq CB\leq \sigma K+2\|t_0\|,
\label{ineg}
\end{equation}
where $K=\underset{\|v\|=1}{\sup} \E\left( \underset{g\in G} \sup \psh{g\cdot v}{\epsilon}\right)$ is a constant which depends only on the standardised noise and on the group action. We have $K\in (0,1]$.
The consistency bias has the following asymptotic behavior when the noise level $\sigma$ tends to infinity:
\begin{equation}
CB=\sigma K+o(\sigma) \mbox{ as } \sigma\to +\infini.
\label{equivalent}
\end{equation}

\end{theorem}
It follows from Equation~\eqref{ineg} that $K$ is the consistency bias with a null template $t_0=0$ and a standardised noise $\sigma=1$.
We can ensure the presence of inconsistency as soon as the signal to noise ratio verifies $\frac{\|t_0\|}{\sigma}<\frac{K}{2}$. Moreover, if the signal to noise ratio verifies $\frac{\|t_0\|}{\sigma}<\frac{K}{3}$ then the consistency bias verifies $CB\geq  \|t_0\|$. In other words, the Fréchet mean in quotient space is too far from the template: the template estimation with the Fréchet mean in quotient space is useless in this case.
In~\cite{dev} the authors also give lower and upper bounds as a function of $\sigma$ but these bounds are less informative than our current bounds. Indeed, in~\cite{dev} the lower bound goes to zero when the template becomes closed to fixed points. This may suggest that the consistency bias was small for this kind of template, which is not the case.
The proof of Theorem~\ref{theo1} is postponed in Appendix~\ref{sec:proof}, the sketch of the proof is the following:
\begin{itemize} \item $K>0$ because the support of $\epsilon$ is not included in the set of fixed points under the action of $G$.
\item $K\leq 1$ is the consequence of the Cauchy-Schwarz inequality.
\item The proof of Inequalities~\eqref{ineg} is based on the triangular inequalities:
\begin{equation}
\|m_\star\|-\|t_0\|\leq CB=\inf_{g\in G}\| t_0-g\cdot m_\star\|\leq \|t_0\|+\|m_\star\|,
\label{inegtri}
\end{equation}
where $m_\star$ minimizes \eqref{var}: having a piece of information about the norm of $m_\star$ is enough to deduce a piece of information about the consistency bias. 
\item The asymptotic Taylor expansion of the consistency bias~\eqref{equivalent} is the direct consequence of inequalities~\eqref{ineg}.
\end{itemize}

Note that Theorem~\ref{theo1} is absolutely not a contradiction with~\cite{kur} where the authors proved the consistency of the template estimation with the Fréchet mean in quotient space for all $\sigma>0$. Indeed their noise was included in the set of constant functions which are the fixed points under their group action. 

One disadvantage of Theorem~\ref{theo1} is that it ensures the presence of inconsistency for $\sigma$ large enough but it says nothing when $\sigma$ is small, in this case one can refer to~\cite{mio} or~\cite{dev}.

\section{Template estimation with the Max-Max Algorithm}
\label{sec:max}
\subsection{Max-Max Algorithm Converges to a Local Minima of the Empirical Variance}
Section~\ref{sec:inc} can be roughly understood as follows: if we want to estimate the template by minimising the Fréchet mean with quotient space then there is a bias. This supposes that we are able to compute such a Fréchet mean. In practice, we cannot minimise the exact variance in quotient space, because we have only a finite sample and not the whole distribution. In this section we study the estimation of the empirical Fréchet mean with the max-max algorithm. We suppose that the group is finite. Indeed, in this case, the registration can always be found by an exhaustive search. In a compact group acting continuously, the registration also exists but is not necessarily computable without approximation. Hence, the numeric experiments which we conduct in Section~\ref{sec:sim} lead to an empirical Karcher mean in a finite number of steps.

If we have a sample: $Y_1,\tp, Y_I$ of independent and identically distributed copies of $Y$, then we define the empirical variance in the quotient space:
\begin{equation}
  F_I(x)=\frac1I \sum_{i=1}^I d^2_Q([x],[Y_i])=\frac{1}{I} \sum_{i=1}^I \underset{g_i\in G}{\min} \|x-g_i\cdot Y_i\|^2=\frac{1}{I} \sum_{i=1}^I \underset{g_i\in G}{\min} \|g_i\cdot x- Y_i\|^2.
    \label{empv}
\end{equation}
The empirical variance is an approximation of the variance, indeed thanks to the law of large number we have $\lim_{I\to \infini} F_I(x)=F(x)$ for all $x\in H$. One element which minimizes globally (respectively locally) $F_I$ is called an empirical Fréchet mean (respectively an empirical Karcher mean). For $x\in H$ and $\underline g\in G^I$: $\underline g=(g_1,\tp, g_I)$ where $g_i\in G$ for all $i\in 1..I$ we define $J$ an auxiliary function by:
\begin{equation*}
 J(x, \underline g)=\frac1I \underset{i=1}{\overset{I}{\sum}} \|x-g_i\cdot Y_i\|^2=
 \frac1I \underset{i=1}{\overset{I}{\sum}} \|g_i^{-1}\cdot x- Y_i\|^2.
\end{equation*}
The max-max algorithms iteratively minimizes the function $J$ in the variable $x\in H$ and in the variable $\underline g\in G^I$:
\begin{algorithm}
\caption{Max-Max algorithm}
\begin{algorithmic}
\REQUIRE A starting point $m_0\in H$, a sample $Y_1,\tp,Y_I$.
\STATE $n=0$.
\WHILE{Convergence is not reached}
\STATE Minimizing $\underline g\in G^I\mapsto J(m_n,\underline g)$: we get $g_i^{n}$ by registering $Y_i$ with respect to $m_n$.
\STATE Minimizing $x\in H\mapsto J(x,\underline g^n)$: we get $m_{n+1}=\frac1I \sum_{i=1}^I g_i^n \cdot Y_i$.
\STATE $n=n+1$.
\ENDWHILE
\STATE $\hat m=m_{n}$
\end{algorithmic}
\end{algorithm}

Note that the empirical variance does not increase at each step of the algorithm since: 
$
    F_I(m_n)=J(m_n,\underline g^n)\geq J(m_{n+1},\underline g^n)\geq J(m_{n+1},\underline g^{n+1})=F_I(m_{n+1})
$.
This algorithm is sensitive to the the starting point. However we remark that $m_1=\frac{1}{I}\sum_{i=1}^I g_i\cdot Y_i$ for some $g_i\in G$, then without loss of generality, we can start from $m_1=\frac{1}{I} \sum_{i=1}^I g_i\cdot Y_i$ for some $g_i\in G$.

\begin{proposition}
As the group is finite, the convergence is reached in a finite number of steps.
\end{proposition}

\begin{proof}
The sequence $(F_I(m_n))_{n\in \N}$ is non-increasing. Moreover the sequence $(m_n)_{n\in \N}$ takes value in a finite set which is:
$
\{\frac1I \sum_{i=1}^I g_i\cdot Y_i,\: g_i\in G\}.
$
Therefore, the sequence $(F_I(m_n))_{n\in\N}$ is stationary. Let $n\in \N$ such that $F_I(m_n)=F_I(m_{n+1})$. 
Hence the empirical variance did not decrease between step $n$ and step $n+1$ and we have:
\begin{equation*}
F_I(m_n)=J(m_n,\underline g_n)=J(m_{n+1},\underline g_n)=J(m_{n+1},\underline{g}_{n+1})=F_I(m_{n+1}),
\end{equation*}
as $m_{n}$ is the unique element which minimizes $m\mapsto J(m,\underline{g}_{n})$ we conclude that $m_{n+1}=m_n$. 
\end{proof}

\begin{figure}[htb]
    \centering
   \begin{tikzpicture}[line cap=round,line join=round,>=triangle 45,x=1cm,y=1cm,scale=0.7]
\draw [color=cqcqcq,, xstep=1.0cm,ystep=1.0cm] (-0.7852537722908061,-1.1625514403292634) grid (8.874746227709274,8.277448559670814);
\draw[->,ultra thick,color=black] (-0.7852537722908061,0.) -- (8.874746227709274,0.);
\foreach \x in {,1.,2.,3.,4.,5.,6.,7.,8.}
\draw[shift={(\x,0)},color=black] (0pt,-2pt);
\draw[color=black] (8.55474622770927,0.08) node [anchor=south west] {$x$};
\draw[->,ultra thick,color=black] (0.,-1.1625514403292634) -- (0.,8.277448559670814);
\foreach \y in {-1.,1.,2.,3.,4.,5.,6.,7.,8.}
\draw[shift={(0,\y)},color=black] (2pt,0pt) -- (-2pt,0pt);
\draw[color=black] (0.1,7.83744855967081) node [anchor=west] {$\underline g$};
\clip(-0.7852537722908061,-1.1625514403292634) rectangle (8.874746227709274,8.277448559670814);
\draw [line width=1.2pt,color=ffqqqq] (1.,0.)-- (1.,2.);
\draw [line width=1.2pt,color=ffqqqq] (1.,1.15) -- (1.18,1.);
\draw [line width=1.2pt,color=ffqqqq] (1.,1.15) -- (0.82,1.);
\draw [dash pattern=on 5pt off 5pt] (1.,2.)-- (0.,2.);
\draw [dash pattern=on 5pt off 5pt] (3.,2.)-- (3.,0.);
\draw [line width=1.2pt,color=ffqqqq] (1.,2.)-- (3.,2.);
\draw [line width=1.2pt,color=ffqqqq] (2.135,2.) -- (2.,1.835);
\draw [line width=1.2pt,color=ffqqqq] (2.135,2.) -- (2.,2.165);
\draw [dash pattern=on 5pt off 5pt] (3.,3.)-- (0.,3.);
\draw [dash pattern=on 5pt off 5pt] (6.524746227709255,6.497448559670791)-- (0.,6.48);
\draw [line width=1.2pt,color=qqffqq,domain=-0.7852537722908061:8.874746227709274] plot(\x,{(--49.-0.*\x)/7.});
\draw [line width=1.2pt,color=qqffqq] (7.,-1.1625514403292634) -- (7.,8.277448559670814);
\draw [line width=2.pt,color=ffqqqq] (6.504746227709255,6.497395075415429)-- (6.524746227709255,6.497448559670791);
\draw [dash pattern=on 5pt off 5pt] (6.524746227709255,6.497448559670791)-- (6.504746227709255,0.);
\draw [line width=1.2pt,color=ffqqqq] (3.,2.)-- (3.,3.);
\draw [line width=1.2pt,color=ffqqqq] (3.,2.65) -- (3.18,2.5);
\draw [line width=1.2pt,color=ffqqqq] (3.,2.65) -- (2.82,2.5);
\draw [line width=1.2pt,color=ffqqqq] (6.524746227709255,6.497448559670791)-- (7.,6.497448559670791);
\draw [line width=1.2pt,color=ffqqqq] (6.882373113854629,6.49744855967079) -- (6.762373113854628,6.347448559670789);
\draw [line width=1.2pt,color=ffqqqq] (6.882373113854629,6.49744855967079) -- (6.762373113854628,6.647448559670791);
\draw [line width=1.2pt,color=ffqqqq] (7.,6.497448559670791)-- (7.,7.);
\draw [line width=1.2pt,color=ffqqqq] (7.,6.883724279835396) -- (7.165,6.748724279835395);
\draw [line width=1.2pt,color=ffqqqq] (7.,6.883724279835396) -- (6.835,6.748724279835395);
\begin{scriptsize}
\draw [fill=xdxdff] (1.,0.) circle (2.5pt);

\draw [fill=qqqqff] (1.,2.) circle (2.5pt);

\draw [fill=xdxdff] (0.,2.) circle (2.5pt);

\draw [fill=qqqqff] (3.,2.) circle (2.5pt);

\draw [fill=xdxdff] (3.,0.) circle (2.5pt);

\draw [fill=qqqqff] (3.,3.) circle (2.5pt);

\draw [fill=xdxdff] (0.,3.) circle (2.5pt);

\draw [fill=qqqqff] (6.524746227709255,6.497448559670791) circle (2.5pt);

\draw [fill=xdxdff] (7.,0.) circle (2.5pt);

\draw [fill=xdxdff] (0.,7.) circle (2.5pt);
\draw [fill=qqqqff] (7.,7.) circle (2.5pt);

\draw [fill=xdxdff] (4.,0.) circle (2.5pt);

\draw [fill=xdxdff] (6.504746227709255,0.) circle (2.5pt);

\draw [fill=xdxdff] (0.,6.48) circle (2.5pt);

\draw [fill=ttqqqq] (4.5,-0.34) circle (1.0pt);
\draw [fill=ttqqqq] (4.8,-0.34) circle (1.0pt);
\draw [fill=ttqqqq] (5.1,-0.34) circle (1.0pt);
\draw [fill=qqqqff] (6.504746227709255,6.497395075415429) circle (2.5pt);
\draw [fill=ttqqqq] (-0.35,5.3) circle (1.0pt);
\draw [fill=ttqqqq] (-0.35,5.) circle (1.0pt);
\draw [fill=ttqqqq] (-0.35,4.7) circle (1.0pt);
\draw [fill=qqqqff] (7.,6.497448559670791) circle (2.5pt);

\draw[color=xdxdff] (0.7,-0.2) node {$m_0$};
\draw[color=xdxdff] (2.7,-0.2) node {$m_1$};
\draw[color=xdxdff] (3.7,-0.2) node {$m_2$};
\draw[color=xdxdff] (6.4,-0.2) node {$m_{n-1}$};
\draw[color=xdxdff] (7.4,-0.2) node {$m_n$};

\draw[color=xdxdff] (-0.3,2.3) node {$\underline g^0$};
\draw[color=xdxdff] (-0.3,3.3) node {$\underline g^1$};
\draw[color=xdxdff] (-0.21,6.5) node {$\underline g^{n-1}$};
\draw[color=xdxdff] (-0.3,7.3) node {$\underline g^n$};

\draw[color=qqqqff] (1,2.3) node {$J(m_0,\underline g^0)$};
\draw[color=qqqqff] (3,1.7) node {$J(m_1,\underline g^0)$};
\draw[color=qqqqff] (3,3.3) node {$J(m_1,\underline g^1)$};

\draw[color=qqqqff] (5,6.0) node {$J(m_{n-1},\underline{g}^{n-1})$};
\draw[color=qqqqff] (7.8,6.0) node {$J(m_n,\underline g^{n-1})$};
\draw[color=qqqqff] (8,7.4) node {$J(m_n,\underline g^n)$};

\end{scriptsize}
\end{tikzpicture}
    \caption{Iterative minimization of the function $J$ on the two axis, the horizontal axis represents the variable in the space $H$, the vertical axis represents the set of all the possible registrations $G^I$. Once the convergence is reached, the point $(m_n,g_n)$ is the minimum of the function $J$ on the two axis in green. Is this point the minimum of $J$ on its whole domain? There are two pitfalls: firstly this point could be a saddle point, it can be avoided with Proposition~\ref{prop:regu}, secondly this point could be a local (but not global), this is discussed in Subsection~\ref{subsec:local}.}
    \label{fig:my_label}
\end{figure}
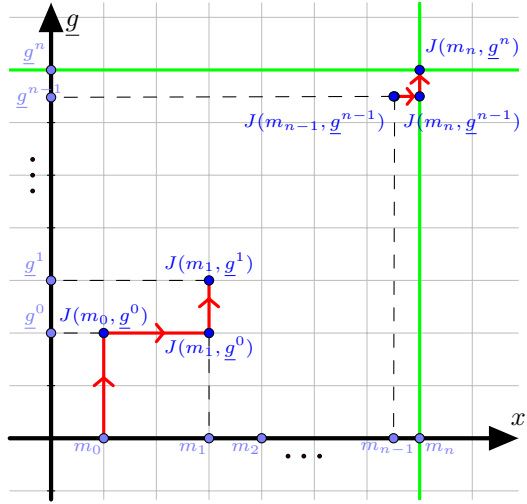

This proposition gives us a shutoff parameter in the max-max algorithm: we stop the algorithm as soon as $m_n=m_{n+1}$.
Let call $\hat m$ the final result of the max-max algorithm. It may seem logical that $\hat m$ is at least a local minimum of the empirical variance. However this intuition may be wrong: let us give a simple counterexample (but not necessarily realistic), suppose that we observe $Y_1,\tp,Y_I$, due to the transformation of the group it is possible that $\sum_{i=1}^n Y_i=0$. We can start from $m_1=0$ in the max-max algorithm, as $Y_i$ and $0$ are already registered, the max-max algorithm does not transform $Y_i$. At step two, we still have $m_2=0$, by induction the max-max algorithm stays at $0$ even if $0$ is not a Fréchet or Karcher mean of $[Y]$. Because $0$ is equally distant from all the points in the orbit of $Y_i$, $0$ is called a focal point of $[Y_i]$. The notion of focal point is important for the consistency of the Fréchet mean in manifold~\cite{bha}. Fortunately, the situation where $\hat m$ is not a Karcher mean is almost always avoided due to the following statement:
\begin{proposition}
\label{prop:regu}
Let $\hat{m}$ be the result of the max-max algorithm. If the registration of $Y_i$ with respect to $\hat{m}$ is unique, in other words, if $\hat m$ is not a focal point of $Y_i$ for all $i\in 1..I$ then
$\hat{m}$ is a local minimum of $F_I$: $[\hat{m}]$ is an empirical Karcher mean of $[Y]$.
\end{proposition}
Note that, if we call $z$ the registration of $y$ with respect to $m$, then the registration is unique if and only if $\psh{m}{z-g\cdot z}\neq 0$ for all $g\in G\setminus\{e\}$. Once the max-max algorithm has reached convergence, it suffices to test this condition for $\hat{m}$ obtained by the max-max algorithm and for $Y_i$ for all $i$. This condition is in fact generic and is always obtained in practice.
\begin{proof}
We call $g_i$ the unique element in $G$ which register $Y_i$ with respect to $\hat m$, for all $h\in G\setminus\{g_i\}$, $\|\hat m-g_i\cdot Y_i\|<\|\hat m-h_i\cdot Y_i\|$. By continuity of the norm we have for $a$ close enough to $m$: $\|a-g_i\cdot Y_i\|<\|a-h_i\cdot Y_i\|$ for all $h_i \neq g_i$ (note that this argument requires a finite group). The registrations of $Y_i$ with respect to $m$ and to $a$ are the same: 
\begin{equation*}
    F_I(a)=\frac1I \sum_{i=1}^I \|a-g_i\cdot Y_i\|^2=J(a,\underline g)\geq J(\hat{m},\underline g)=F_I(\hat{m}),
\end{equation*} 
because $m\mapsto J(m,\underline g)$ has one unique local minimum $\hat{m}$.
\end{proof}

\subsection{Max-Max Algorithm is a Gradient Descent of the Variance}
In this Subsection, we see that the max-max algorithm is in fact a gradient descent. The gradient descent is a general method to find the minimum of a differentiable function. Here we are interested in the minimum of the variance $F$: let $m_0\in H$ and we define by induction the gradient descent of the variance $m_{n+1}=m_n-\rho \nabla F(m_n)$, where $\rho>0$ and $F$ the variance in the quotient space. In~\cite{dev} the gradient of the variance in quotient space for $m$ a regular point was computed ($m$ is regular as soon as $g\cdot m=m$ implies $g=e$), this leads to:
\begin{equation*}
    m_{n+1}=m_n-2\rho\left[m_n-\E(g(Y,m_n)\cdot Y)\right],
\end{equation*}
where $g(Y,m_n)$ is the almost-surely unique element of the group which register $Y$ with respect to $m_n$. Now if we have a set of data $Y_1,\tp, Y_n$ we can approximated the expectation which leads to the following approximated gradient descent:
\begin{equation*}
    m_{n+1}= m_n(1-2\rho) +\rho \frac2I \sum_{i=1}^I g(Y_i,m_n)\cdot Y_i,
\end{equation*}
now by taking $\rho=\frac12$ we get $m_{n+1}= \frac1I \sum_{i=1}^I g(Y_i,m_n)\cdot Y_i$. So the approximated gradient descent with $\rho=\frac12$ is exactly the max-max algorithm. But the max-max algorithm is proven to be converging in a finite number of steps which is not the case for gradient descent in general.

\section{Simulation on synthetic data}
\label{sec:sim}

In this Section\footnote{The code used in this Section is available at \href{http://loic.devilliers.free.fr/ipmi.html}{http://loic.devilliers.free.fr/ipmi.html}.}, we consider data in an Euclidean space $\R^N$ equipped with its canonical dot product $\psh{\cdot }{\cdot}$, and $G=\Z/\N\Z$ acts on $\R^N$ by circular permutation on coordinates:
\begin{equation*}
(\bar{k}\in \Z/N\Z, (x_1,\tp,x_N)\in \R^N)\mapsto (x_{1+k},x_{2+k},\tp x_{N+k}),
\end{equation*}
where indexes are taken modulo $N$. This space models the discretization of functions with $N$ points. This action is found in~\cite{all} and used for neuroelectric signals in~\cite{hit}. The registration between two vectors can be made by an exhaustive research but it is faster with the fast Fourier transform~\cite{fft}.  

\subsection{Max-max algorithm with a step function as template}

\begin{figure}[!htb]
\centering
\subfigure[Example of a template (a step function) and the template estimation with a sample size $10^5$ in $\R^{64}$, $\epsilon$ is Gaussian noise and $\sigma=10$. At the discontinuity points of the template, we observe a Gibbs-like 
    phenomena.]{\includegraphics[width=5.6cm]{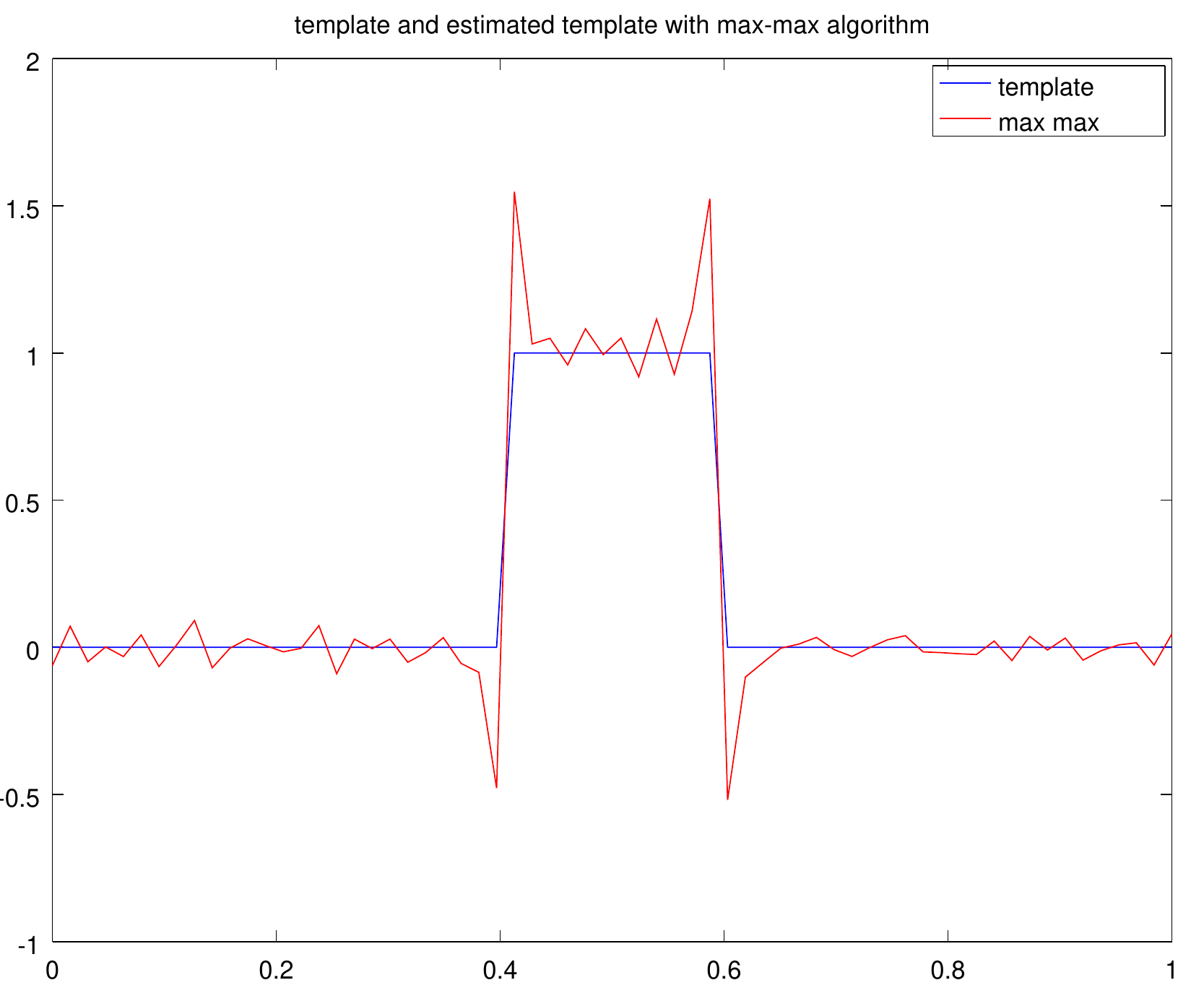} \label{fig:bias}}
\quad
\subfigure[Variation of $F_I(t_0)$ (in blue) and of $F_I(\hat m)$ (in red) as a function of $I$ the size of the sample. 
Since convergence is already reached, $F(\hat m)$, which is the limit of red curve, is below $F(t_0)$: $F(t_0)$ is the limit of the blue curve. Due to the inconsistency, $\hat m$ is an example of point such that $F(\hat m)<F(t_0)$. ]{\includegraphics[width=5.8cm]{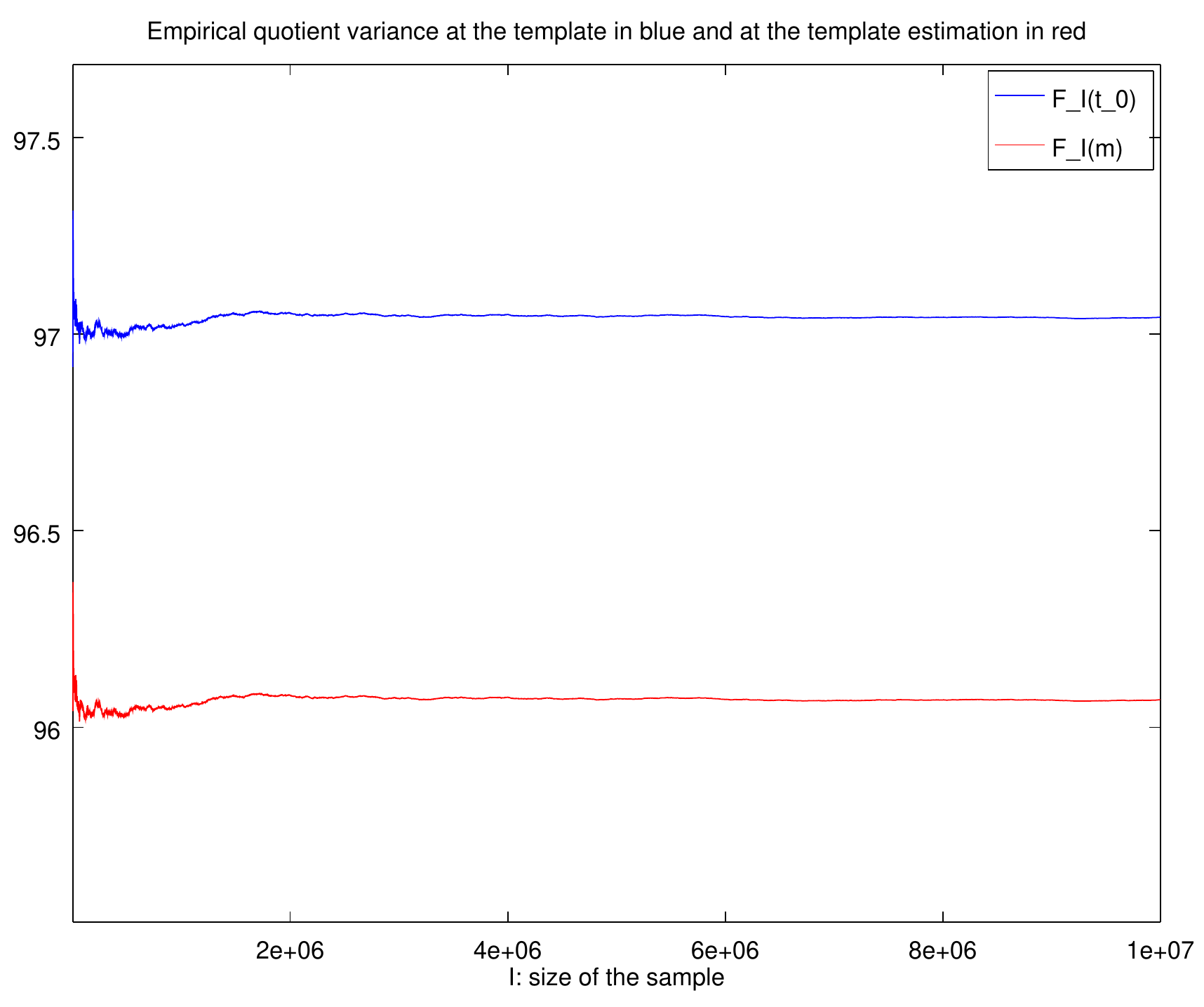}\label{fig:variance}}
\caption{Template $t_0$ and template estimation $\hat m$ on Fig.~\ref{fig:bias}. Empirical variance at the template and the template estimation with the max-max algorithm as a function of the size of the sample on Fig.~\ref{fig:variance}.}

\end{figure}
We display an example of a template and the template estimation with the max-max algorithm on Fig~\ref{fig:bias}.
Note that this experiment was already conducted in~\cite{all}. But no explanation of the appearance of the bias was provided.
On the opposite, we know from the precedent Section that the max-max result is an empirical Karcher mean, and that this result can be obtained in a finite number of steps. Taking $\sigma=10$ may seem extremely high, however the standard deviation of the noise at each point is not $10$ but $\frac{\sigma}{\sqrt N}=1.25$ which is not so high.

The sample size is $10^5$, and the algorithm stopped after 94 steps, and $\hat m$ the estimated template (in red on the Fig.~\ref{fig:bias}) is not a focal points of the orbits $[Y_i]$, then Proposition~\ref{prop:regu} applies. We call empirical bias (noted EB) the quotient distance between the true template and the point $\hat m$ given by the max-max result. On this experiment we have $\frac{EB}{\sigma}\simeq 0.11$. Of course, one could think that we estimate the template with an empirical bias due to a too small sample size which induces fluctuation. To reply to this objection, we keep in memory~$\hat{m}$ obtained with the max-max algorithm. If there was no inconsistency then we would have $F(t_0)\leq F(\hat{m})$. We do not know the value of the variance $F$ at these points, but thanks to the law of large number, we know that:
\begin{equation*}
F(t_0)=\underset{I\to \infini}{\lim} F_I(t_0) \mbox{ and } F(\hat{m})=\underset{I\to \infini}{\lim} F_I(\hat{m}),
\end{equation*}
Given a sample, we compute $F_I(t_0)$ and $F_I(\hat{m})$ thanks to the definition of the empirical variance $F_I$~\eqref{empv}. We display the result on Fig.~\ref{fig:variance}, this tends to confirm that $F(t_0)>F(\hat{m})$. In other words, the variance at the template is bigger that the variance at the point given by the max-max algorithm. 

\subsection{Max-max algorithm with a continuous template}
\begin{figure}[htb]
    \centering
    \includegraphics[scale=0.35]{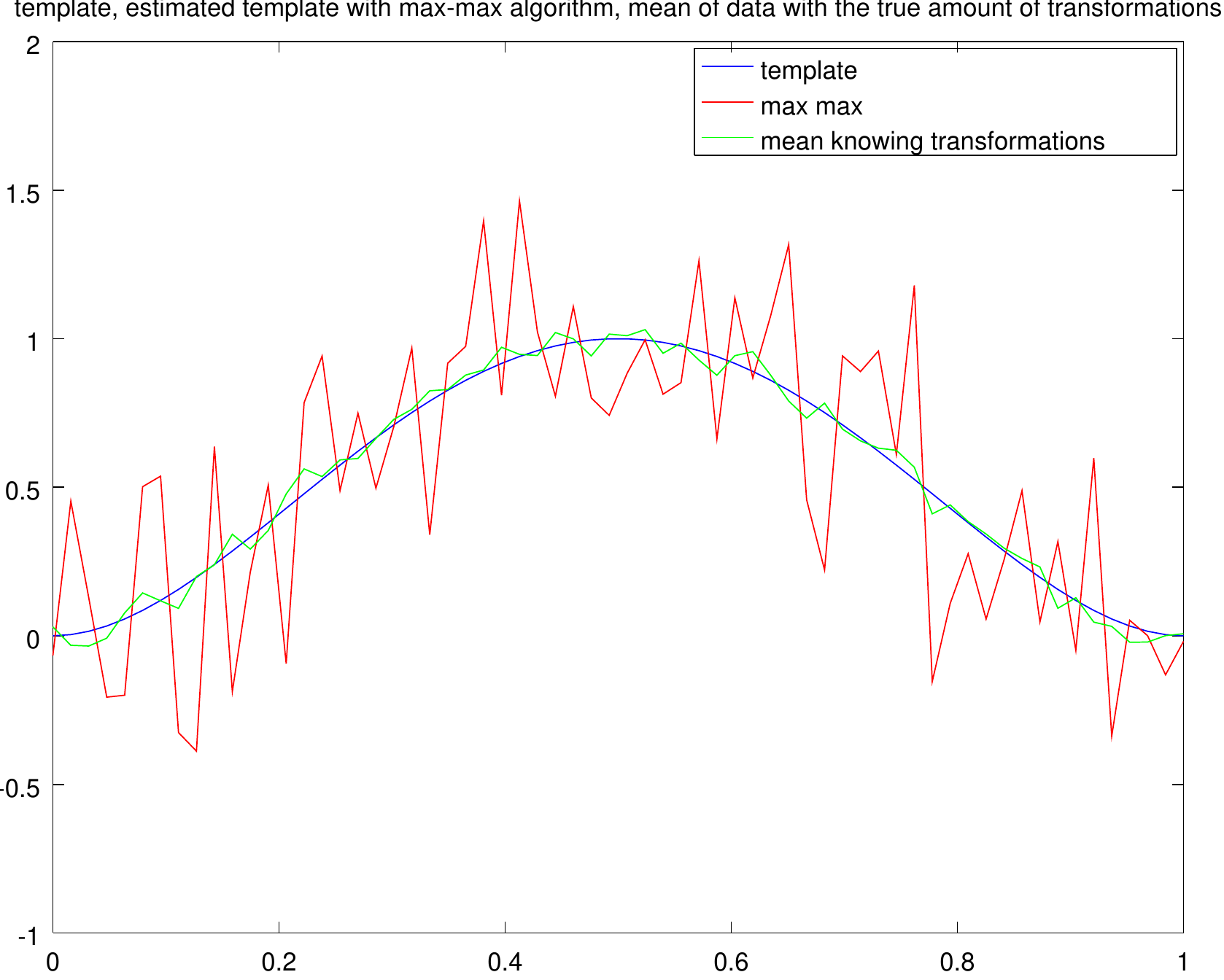}
    \caption{Example of an other template (here a discretization of a continuous function) and the template estimation with a sample size $10^3$ in $\R^{64}$ (in red), $\epsilon$ is Gaussian noise and $\sigma=10$. Even with a continuous function the inconsistency appears. In green we compute the mean of data with the true amount of transformations.}
    \label{fig:bias3}
\end{figure}
Figure~\ref{fig:bias} shows that the main source of the inconsistency was the discontinuity of the template. We could think that a continuous template leads to consistency. But it is not the case, even with a large number of observations created from a continuous template we do not observe a convergence to the template see Fig.~\ref{fig:bias3}, the empirical bias satisfies $\frac{EB}{\sigma}=0.25$. If we knew the original transformations we could invert the transformations on data and take the mean, that is what we deed in green on Fig.~\ref{fig:bias3}. We see that with a sample size $10^3$, the mean gives us almost the good result since we have in that case $\frac{EB}{\sigma}=0.03$.

\subsection{Does the max-max algorithm give us a global minimum or only a local minimum of the variance?}
\label{subsec:local}
Proposition~\ref{prop:regu} tells us that the output of the max-max algorithm is a Karcher mean of the variance, but we do not know that if it is Fréchet mean of the variance. In other words, is the output a global minimum of the variance? 
In fact, $F_I$ has a lot of local minima which are not global. Indeed we can use the max-max algorithm with different starting points and we observe different outputs (which are all local minima thanks to Proposition~\ref{prop:regu}) with different empirical variance (result non shown).

\section{Discussion and Conclusion}

We provided an asymptotic behavior of the consistency bias when the noise level $\sigma$ tends to infinity, as a consequence, the inconsistency cannot be neglected when $\sigma$ is large. However we have not answered this question: can the inconsistency be neglected? When the noise level is small enough, then the consistency bias is small~\cite{mio,dev}, hence it can be neglected. Note that the quotient space is not a manifold, this prevents us to use \textit{a priori} the Central Limit theorem for manifold proved in~\cite{bha}. But if the Central Limit theorem could be applied to quotient space, the fluctuations induce an error which would be approximately equal to $\frac{\sigma}{\sqrt I}$ and if $K\ll\frac{1}{\sqrt I}$, then the inconsistency could be neglected because it is small compared to fluctuation. 

If the Hilbert Space is a functional space, for instance $L^2([0,1])$, in practice, we never observe the whole function, only a finite number values of this function. One can model these observable values on a grid. When the resolution of the grid goes to zero, one can show the consistency~\cite{pan} by using the Fréchet mean with the Wasserstein distance on the space of measures rather than in the space of functions. But in (medical) images the number of pixels or voxels is finite.

Finally, in a future work one needs to study the template estimation with non isometric action. But we can already learn from this work: in the numerical experiments we led, we have seen that the template estimated is more detailed that the true template. The intuition is that the estimated template in computational anatomy with a group of diffeomorphisms is also more detailed. But the true template is almost always unknown. It is then possible that one think that the computation of the template succeeded to capture small details of the template while it is just an artifact due to the inconsistency. Moreover in order to tackle this question, one needs to have a good modelisation of the noise, for instance in~\cite{kur}, the observations are curves, what is a relevant noise in the space of curves?

\bibliographystyle{alpha}
\bibliography{bi}

\appendix
\section{Proof of Theorem~\ref{theo1}}
\label{sec:proof}

\begin{proof}
In the proof, we note by $S$ the unit sphere in $H$. In order to prove that $K>0$, we take $x$ in the support of $\epsilon$ such that $x$ is not a fixed point under the action of $G$. It exists $g_0\in G$ such that $g_0\cdot x\neq x$. We note $v_0=\frac{g_0\cdot x}{\|x\|}\in S$, we have $\psh{v_0}{g_0\cdot x}=\|x\|>\psh{v_0}{x}$ and by continuity of the dot product it exists $r>0$ such that: $
\forall y\in B(x,r)\quad \psh{v_0}{g_0\cdot y}>\psh{v_0}{y}
$ as $x$ is in the support of $\epsilon$ we have $\P(\epsilon \in B(x,r))>0$, it follows:
\begin{equation}
\P\left( \underset{g\in G}{\sup} \psh{v_0}{g\cdot \epsilon}>\psh{v_0}{\epsilon}\right)>0.
\label{inegv0}
\end{equation}
Thanks to Inequality~\eqref{inegv0} and the fact that $\sup_{g\in G} \psh{v_0}{g\cdot \epsilon}\geq \psh{v_0}{\epsilon}$ we have:
\begin{equation*}
K=\underset{v\in S}{\sup}\: \E\left(\underset{g\in G}{\sup} \psh{v}{g \cdot \epsilon}\right)\geq \E\left(\underset{g\in G}{\sup} \psh{v_0}{g \cdot \epsilon}\right)>\E(\psh{v_0}{\epsilon})=\psh{v_0}{\E(\epsilon)}=0.
\end{equation*}
Using the Cauchy-Schwarz inequality:
$K\leq \sup_{v\in S} \E(\|v\|\times \|\epsilon\|)\leq \E(\| \epsilon\|^2)^{\frac12}=1$.
We now prove Inequalities~\eqref{ineg}. The variance at $\lambda v$ for $v\in S$ and $\lambda\geq 0$ is:
\begin{equation}
F(\lambda v)=\E\left( \underset{g\in G}{\inf} \|\lambda v-g \cdot Y\|^2\right)=\lambda^2-2\lambda \E\left( \underset{g\in G}{\sup} \psh{v}{g \cdot Y}\right) +\E(\|Y\|^2).
 \label{expand}
\end{equation}
Indeed $\|g\cdot Y\|=\|Y\|$ thanks to the isometric action. We note $x^+=\max(x,0)$ the positive part of $x$ and  $h(v)=\E(\sup_{g\in G}\psh{v}{g \cdot Y})$. The $\lambda\geq 0$ which\footnote{Indeed we know that $x\in \R^+\mapsto x^2-2bx+c$ reaches its minimum at the point $x=b^+$ and $f(b^+)=c-(b^+)^2$.} minimizes~\eqref{expand} is $h(v)^+ $
and the minimum value of the variance restricted to the half line $\R^+v$ is $F(h(v)^+ v)=\E(\|Y\|^2)- (h(v)^+)^2$.
To find $[m_\star]$ the Fréchet mean of~$[Y]$, we need to maximize $(h(v)^+)^2$ with respect to $v\in S$: $m_\star=h(v_\star)v_\star$ with\footnote{Note that we remove the positive part and the square because $\aM \: h=\aM \: (h^+)^2$ since $h$ takes a non negative value (indeed $h(v)\geq \E(\psh{v}{\phi\cdot t_0+\epsilon})=\psh{v}{\E(\phi\cdot t_0)}$ and this last quantity is non negative for at least one $v\in S$).} $v_\star\in \aM_{v\in S} \:h(v)$. As we said in the sketch of the proof we are interested in getting a piece of information about the norm of $\|m_\star\|$ we have:
$
\|m_\star\|=h(v_\star)=\sup_{v\in S} h.
$
Let $v\in S$, we have: $-\|t_0\|\leq \psh{v}{g\phi\cdot t_0}\leq  \|t_0\|$ because the action is isometric. Now we decompose $Y=\phi\cdot t_0+\sigma \epsilon$ and we get:
\begin{eqnarray*}
h(v)&=&\E\left(\underset{g\in G}{\sup} \psh{v}{g\cdot Y}\right)=
\E\left(\underset{g\in G}{\sup} \left(\psh{v}{g\cdot \sigma\epsilon}+\psh{v}{g\phi \cdot t_0 }\right)\right)\\
h(v)&\leq& \E\left(\underset{g\in G}{\sup} \left(\psh{v}{g \cdot \sigma\epsilon}+\|t_0\|\right) \right)  = \sigma\E\left(\underset{g\in G}{\sup} \psh{v}{g \cdot \epsilon}\right) +\|t_0\|\\
h(v)&\geq& \E\left(\underset{g\in G}{\sup} \left(\psh{v}{g\cdot \sigma\epsilon}\right) -\|t_0\|\right)  = \sigma \E\left(\underset{g\in G}{\sup} \psh{v}{g \cdot \epsilon}\right)-\|t_0\|.
\end{eqnarray*}
By taking the biggest value in these inequalities with respect to $v\in S$, by definition of $K$ we get:
\begin{equation}
-\|t_0\| +\sigma K \leq \|m_\star\| \leq \|t_0\|+\sigma K. \label{majthree}
\end{equation}
Thanks to~\eqref{majthree} and to~\eqref{inegtri}, Inequalities~\eqref{ineg} are proved.
\end{proof}

\end{document}